\documentclass[11pt]{article}
\usepackage[letterpaper,hmargin=1.1in,vmargin=1in]{geometry}
\usepackage[numbers,compress,sort]{natbib}
\usepackage[T1]{fontenc}
\usepackage[utf8]{inputenc}
\usepackage{lmodern}
\usepackage{amsmath,amsfonts,amsthm}
\usepackage{aliascnt}
\usepackage{microtype}
\usepackage{xcolor}
\usepackage{tikz}
\usetikzlibrary{arrows.meta,calc,decorations.pathreplacing}
\usepackage[font=small,labelfont=bf,labelsep=period]{caption}
\definecolor{singleton}{HTML}{237F87}
\definecolor{overlap}{HTML}{B8752C}
\definecolor{ink}{HTML}{26323B}
\definecolor{guide}{HTML}{838B91}
\definecolor{link}{HTML}{1F4E79}
\usepackage[colorlinks=true,linkcolor=link,citecolor=link,urlcolor=link]{hyperref}
\usepackage[capitalize,noabbrev]{cleveref}

\newcommand{\R}{\mathbb R}
\newcommand{\Z}{\mathbb Z}
\newcommand{\x}{\mathbf{x}}
\newcommand{\z}{\mathbf{z}}
\DeclareMathOperator{\conv}{conv}
\DeclareMathOperator{\intr}{int}
\newtheorem{theorem}{Theorem}
\newaliascnt{lemma}{theorem}
\newtheorem{lemma}[lemma]{Lemma}
\aliascntresetthe{lemma}
\crefname{lemma}{Lemma}{Lemmas}
\title{A centerpoint theorem for three planar convex bodies}
\author{Hongyu Cheng\thanks{Department of Applied Mathematics and Statistics, Johns Hopkins University, Baltimore, MD 21218, USA. E-mail addresses: \texttt{hongyucheng@jhu.edu}, \texttt{basu.amitabh@jhu.edu}.} \and Amitabh Basu\footnotemark[1]}
\date{}

\begin{document}
\maketitle

\begin{abstract}
For planar convex bodies $A_0,A_1,A_2$ satisfying $\frac12(A_0+A_2)\subseteq A_1$, we prove that the union of the three slices $\{i\}\times A_i\subseteq\R^3$, $i=0,1,2$, contains a point such that every closed halfspace containing it captures at least $2/9$ of their total area. This establishes the three-slice case of Oertel's mixed-integer centerpoint conjecture in $\Z\times\R^2$ with the best possible constant.
\end{abstract}

\section{Introduction}\label{sec:intro}

Gr\"unbaum's inequality~\cite{Grunbaum} guarantees the fraction $(d/(d+1))^d$ of the volume of a convex body (nonempty compact convex set) in every halfspace containing its centroid in $\R^d$. Write $|\cdot|$ for $d$\nobreakdash-dimensional volume and $\mu(E)=\sum_{\z\in\Z^n}|\{\x\in\R^d:(\z,\x)\in E\}|$ for the mixed-integer volume of a Borel set $E\subseteq\R^{n+d}$. Oertel~\cite[Conjecture~4.1.20]{Oertel} conjectured that every compact convex set $C\subseteq\R^{n+d}$ with $\mu(C)>0$ contains a point in $\Z^n\times\R^d$ such that every closed halfspace $H$ containing it satisfies
\[
\mu(C\cap H)\ge2^{-n}\left(\frac{d}{d+1}\right)^d\mu(C).
\]
This fraction, equal to $2/9$ for $n=1$ and $d=2$, controls the oracle complexity of centerpoint-based cutting plane methods for mixed-integer convex minimization~\cite{BasuOertel}.

A Helly-type argument gives the fraction $1/(2^n(d+1))$ for every $C$~\cite[Corollary~3.4]{BasuOertel}, but falls short of the conjectured fraction for $d\ge2$. Basu and Oertel~\cite[Theorem~3.6]{BasuOertel} proved the conjecture when the projection of $C$ onto $\R^n$ has lattice width above a threshold exponential in $n$, while Cristi and Salas~\cite{CristiSalas} obtained a polynomial threshold, requiring a Euclidean ball of radius $\Omega(d^2n^{3/2})$ in the projection, or an interval of length linear in $d$ when $n=1$. Cheng and Basu~\cite{ChengBasu} proved that an $\ell_\infty$ ball of radius $k\ge\frac{3e}2(n+d)$ in the projection guarantees the fraction $1/e-3(n+d)/(2k)$, whereas a radius $o(n+d)$ guarantees no dimension-independent positive fraction. The results above that reach the conjectured fraction all require the projection of $C$ onto $\R^n$ to be large, which for $n=1$ means that $C$ has many slices $C\cap(\{z\}\times\R^d)$, $z\in\Z$. Much less is known when the slices are few. The cases of one or two slices of positive volume follow from Gr\"unbaum's inequality on the largest slice, which carries at least half of the mixed-integer volume, but to the best of our knowledge the conjecture was open already for three slices in $\Z\times\R^2$.

We prove the three-slice case with the sharp constant $2/9$. To state the result, let $A_0,A_1,A_2\subseteq\R^2$ be convex bodies, scaled so that $|A_0|+|A_1|+|A_2|=1$. Put $S=\bigcup_{i=0}^2(\{i\}\times A_i)$, call $\mu(S\cap E)$ the {\em mass} of a Borel set $E\subseteq\R^3$ (\cref{fig:setup}), and define the {\em halfspace depth} $h_S(\mathbf y)=\inf_{H\ni\mathbf y}\mu(S\cap H)$ for $\mathbf y\in\R^3$, where $H$ ranges over closed halfspaces in $\R^3$.

\begin{theorem}\label{thm:main}
Let $A_0,A_1,A_2$ be planar convex bodies with $|A_0|+|A_1|+|A_2|=1$ and $\frac12(A_0+A_2)\subseteq A_1$. Then there exists $\mathbf y\in S$ such that $h_S(\mathbf y)\ge2/9$. In particular, if $C\subseteq\R^3$ is a compact convex set whose slices $C\cap(\{z\}\times\R^2)$ have positive area for exactly three integers $z$, then $C\cap(\Z\times\R^2)$ contains a point such that every closed halfspace $H$ containing it satisfies $\mu(C\cap H)\ge\frac29\mu(C)$.
\end{theorem}

For sharpness, take $A_1=A_2=\Delta$, a triangle of area $(1-\varepsilon)/2$ with $0<\varepsilon<1/3$, and let $A_0$ be the homothetic copy of $\Delta$ about its centroid with area $\varepsilon$. The midpoint inclusion $\frac12(A_0+A_2)\subseteq A_1$ follows from $A_0\subseteq\Delta$. Every point $\x\in\Delta$ has a barycentric coordinate of at least $1/3$, and the line through $\x$ parallel to the opposite side cuts off a triangle of area at most $(2/3)^2|\Delta|=\frac49|\Delta|$ at the corresponding vertex. For $i\in\{1,2\}$, tilting a plane about this line gives a closed halfspace through $(i,\x)$ that meets $\{i\}\times\Delta$ in this triangle and misses $\{3-i\}\times\Delta$, hence has mass at most $\frac49|\Delta|+\varepsilon=2/9+7\varepsilon/9$. The halfspace $\{\mathbf y\in\R^3:y_1\le1/2\}$ meets $S$ only in $\{0\}\times A_0$, so the points of this slice have depth at most $\varepsilon$, which proves sharpness as $\varepsilon\downarrow0$. The mixed-integer conclusion is also sharp, since $C=\conv S$ satisfies $C\cap(\Z\times\R^2)=S$.

The proof of \cref{thm:main} rests on a planar configuration that Gr\"unbaum introduced for a different purpose. In his survey of measures of symmetry~\cite{GrunbaumSymmetry}, he conjectured that when three concurrent lines divide a planar convex body into six sectors whose areas alternate between two values, the smaller value is at least half the larger. Lillington~\cite{Lillington} proved this bound, and Tsintsifas~\cite{Tsintsifas} proved the stronger conjecture of Gardner, Kwapie\'n and Laurie~\cite{GardnerKwapienLaurie} that, for arbitrary sector areas, the ratios of three pairwise nonadjacent sector areas to the opposite ones sum to at least $3/2$. Equality holds for a triangle and the lines through its centroid parallel to its sides, each cutting off $4/9$ of the area, the extremal fraction in Gr\"unbaum's inequality for $d=2$.

In our proof, these three lines come from Helly's theorem. If every point of the middle slice lies in a halfspace of mass less than $2/9$, then at most three such halfspaces already cover this slice. When no two of them suffice, one of them can be shrunk so that their boundaries meet the slice along three concurrent lines. The six sectors are then covered alternately by one and by two of the halfspaces, and Tsintsifas's inequality, with the doubly covered sectors in the numerators, shows that for two of the halfspaces the area of the middle slice covered by both is at least $9/2$ times the square of the area left uncovered (\cref{lem:covering}). Since these two halfspaces have total mass less than $4/9$, they must leave large portions of the outer slices uncovered, which contradicts the Brunn--Minkowski inequality applied to the uncovered portions of the three slices.

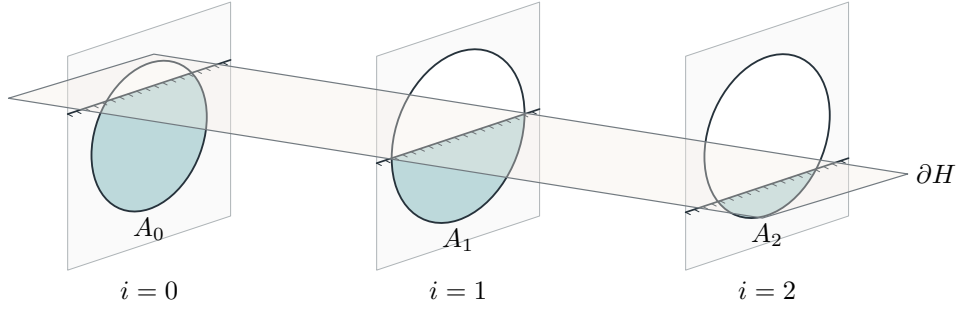
\begin{figure}[t]
\centering
\begin{tikzpicture}[x=1.15cm,y=1.05cm,font=\small,line cap=round,line join=round]
  \def\sp{3.55}
  \def\kx{0.72}
  \def\ky{0.24}
  \def\lam{0.62}
  \foreach \t/\r in {0/0.92,1/1.06,2/1.00}{
    \begin{scope}[cm={\kx,\ky,0,1,(\sp*\t,0)}]
      \fill[black!2] (-1.3,-1.35) rectangle (1.3,1.35);
      \draw[guide!65,line width=0.35pt] (-1.3,-1.35) rectangle (1.3,1.35);
      \fill[white] (0,0) circle[radius=\r];
      \begin{scope}
        \clip (0,0) circle[radius=\r];
        \fill[singleton!30] (-2,-2) rectangle (2,{\lam*(1-\t)});
      \end{scope}
      \draw[ink,line width=0.7pt] (0,0) circle[radius=\r];
      \draw[ink,line width=0.65pt,postaction={decorate,draw=ink,line width=0.3pt,
        decoration={border,angle=-45,amplitude=0.7mm,segment length=1.2mm}}]
        (-1.3,{\lam*(1-\t)}) -- (1.3,{\lam*(1-\t)});
      \node[anchor=north,inner sep=2pt] at (0,{-\r-0.04}) {$A_{\t}$};
    \end{scope}
    \node at (\sp*\t,-1.95) {$i=\t$};
  }
  \path[fill=overlap!12,fill opacity=0.35,draw=ink!65,line width=0.45pt]
    ({-0.22*\sp-1.16*\kx},{1.22*\lam-1.16*\ky}) --
    ({-0.22*\sp+1.16*\kx},{1.22*\lam+1.16*\ky}) --
    ({2.22*\sp+1.16*\kx},{-1.22*\lam+1.16*\ky}) --
    ({2.22*\sp-1.16*\kx},{-1.22*\lam-1.16*\ky}) -- cycle;
  \node[anchor=west,inner sep=3pt] at ({2.22*\sp+1.16*\kx},{-1.22*\lam+1.16*\ky}) {$\partial H$};
\end{tikzpicture}
\caption{The mass $\mu(S\cap H)$ of a halfspace $H$ is the total area of the shaded portions of $S$.}
\label{fig:setup}
\end{figure}

\section{A planar covering lemma}\label{sec:planar}

\begin{lemma}\label{lem:covering}
Let $A\subseteq\R^2$ be a convex body, let $\beta>0$, and let $G^1,\ldots,G^m\subseteq\R^2$ be open halfplanes covering $A$ with $|A\cap G^j|\le\beta$ for $j=1,\ldots,m$. Then there are indices $r,s\in\{1,\ldots,m\}$, possibly equal, such that $\sigma^2\le\beta\tau$, where $\sigma=|A\setminus(G^r\cup G^s)|$ and $\tau=|A\cap G^r\cap G^s|$.
\end{lemma}

\begin{proof}
If two of the halfplanes, possibly equal, cover $A$ up to a set of area zero, then $\sigma=0$ and the conclusion holds. Assume from now on that no two of the halfplanes cover $A$ up to a set of area zero. Since the halfplanes cover $A$, the convex sets $A$ and $\R^2\setminus G^j$, $j=1,\ldots,m$, have empty intersection, and Helly's theorem gives a minimal subfamily of at most three of them with empty intersection. This subfamily has three members and does not contain $A$, since otherwise at most two halfplanes would cover $A$. Thus, after relabeling, we obtain three halfplanes $G^1,G^2,G^3$ that cover the plane, whereas no two do so.

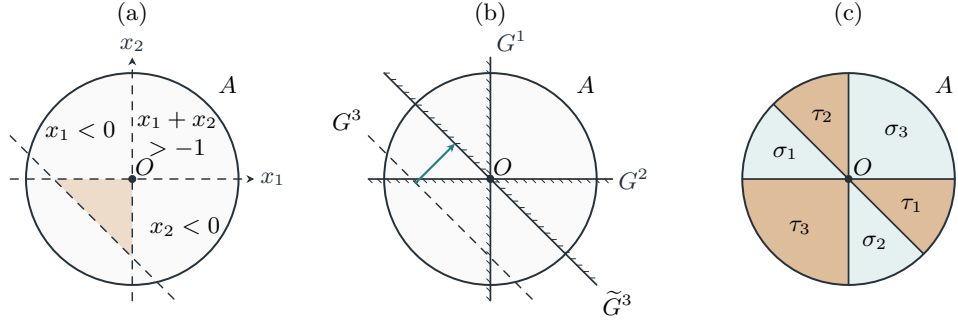
\begin{figure}[t]
\centering
\begin{tikzpicture}[x=1.04cm,y=1.04cm,font=\footnotesize,line cap=round,line join=round,
  original/.style={ink,dashed,line width=0.55pt},
  moved/.style={ink,line width=0.65pt},
  hatch/.style={postaction={decorate,draw=ink,line width=0.3pt,
    decoration={border,angle=#1,amplitude=0.7mm,segment length=1.2mm}}},
  motion/.style={-{Stealth[length=3.2pt,width=3pt]},singleton,line width=0.8pt}]
  \begin{scope}
    \node at (0,2.10) {(a)};
    \filldraw[fill=black!2,draw=ink,line width=0.75pt] (0,0) circle[radius=1.35];
    \node[above right,inner sep=2pt] at (1.02,1.02) {$A$};
    \fill[overlap!25] (0,0) -- (-1,0) -- (0,-1) -- cycle;
    \draw[original,-{Stealth[length=3pt,width=3pt]}] (-1.55,0) -- (1.55,0) node[right,inner sep=2pt] {$x_1$};
    \draw[original,-{Stealth[length=3pt,width=3pt]}] (0,-1.55) -- (0,1.55) node[above,inner sep=2pt] {$x_2$};
    \draw[original] (-1.55,0.55) -- (0.55,-1.55);
    \fill[ink] (0,0) circle[radius=1.5pt];
    \node[inner sep=0.5pt] at (0.16,0.16) {$O$};
    \node at (-0.66,0.61) {$x_1<0$};
    \node at (0.67,-0.61) {$x_2<0$};
    \node[align=center] at (0.57,0.57) {$x_1+x_2$\\$>-1$};
  \end{scope}
  \begin{scope}[shift={(4.55,0)}]
    \node at (0,2.10) {(b)};
    \filldraw[fill=black!2,draw=ink,line width=0.75pt] (0,0) circle[radius=1.35];
    \node[above right,inner sep=2pt] at (1.02,1.02) {$A$};
    \draw[original] (-1.55,0.55) -- (0.55,-1.55);
    \node[anchor=south east,inner sep=2pt] at (-1.55,0.55) {$G^3$};
    \draw[moved,hatch=45] (0,-1.55) -- (0,1.55) node[above right,inner sep=2pt] {$G^1$};
    \draw[moved,hatch=-45] (-1.55,0) -- (1.55,0) node[right,inner sep=2pt] {$G^2$};
    \draw[moved,hatch=45] (-1.35,1.35) -- (1.35,-1.35);
    \draw[motion] (-0.95,-0.05) -- (-0.45,0.45);
    \node[below right,inner sep=2pt] at (1.35,-1.35) {$\widetilde G^3$};
    \fill[ink] (0,0) circle[radius=1.5pt];
    \node[inner sep=0.5pt] at (0.16,0.16) {$O$};
  \end{scope}
  \begin{scope}[shift={(9.1,0)}]
    \node at (0,2.10) {(c)};
    \fill[singleton!12] (0,0) circle[radius=1.35];
    \foreach \a/\b in {90/135,180/270,315/360}
      \fill[overlap!48] (0,0) -- (\a:1.35) arc[start angle=\a,end angle=\b,radius=1.35] -- cycle;
    \draw[moved] (0,-1.35) -- (0,1.35) (-1.35,0) -- (1.35,0) (135:1.35) -- (315:1.35);
    \draw[ink,line width=0.75pt] (0,0) circle[radius=1.35];
    \foreach \a/\label in {45/{\sigma_3},112.5/{\tau_2},157.5/{\sigma_1},225/{\tau_3},292.5/{\sigma_2},337.5/{\tau_1}}
      \node at (\a:0.87) {$\label$};
    \node[above right,inner sep=2pt] at (1.02,1.02) {$A$};
    \fill[ink] (0,0) circle[radius=1.5pt];
    \node[inner sep=0.5pt] at (0.18,0.16) {$O$};
  \end{scope}
\end{tikzpicture}
\caption{(a) The halfplanes $G^1,G^2,G^3$ for $a=b=1$, with their common intersection shaded. (b) The boundary of $G^3$ moves to pass through $O$, which gives $\widetilde G^3$. (c) Opposite sectors have areas $\sigma_j$ and $\tau_j$.}
\label{fig:sectors}
\end{figure}

Their boundaries cannot all be parallel, since then two of them would already cover the plane. After relabeling and an area-preserving affine change of coordinates, take $G^1=\{\x\in\R^2:x_1<0\}$ and $G^2=\{\x\in\R^2:x_2<0\}$, with their boundaries meeting at the origin $O$ (\cref{fig:sectors}(a)). The halfplane $G^3$ contains the nonnegative quadrant and therefore has the form $\{\x\in\R^2:a x_1+b x_2>-1\}$ with $a,b\ge0$, both strictly positive because a zero coefficient would allow two halfplanes to cover the plane. Set $\widetilde G^3=\{\x\in\R^2:a x_1+b x_2\ge0\}\subseteq G^3$ (\cref{fig:sectors}(b)), which still contains the nonnegative quadrant, so that $G^1,G^2,\widetilde G^3$ cover the plane. Since no two of $G^1,G^2,G^3$ cover $A$ up to a set of area zero, for each $j=1,2,3$ we can choose a point $\mathbf p_j\in A$ that lies in neither of the other two halfplanes nor on a coordinate axis. Then $\mathbf p_3$ lies in the open first quadrant, while $\mathbf p_1$ and $\mathbf p_2$ lie outside $G^3$, in the open second and fourth quadrants respectively. The segment joining $\mathbf p_1$ and $\mathbf p_2$ therefore lies in $\{\x\in\R^2:a x_1+b x_2<0\}$ and crosses both negative coordinate axes, so the ray from $\mathbf p_3$ through $O$ meets the relative interior of this segment beyond $O$. Thus $O$ lies in the interior of $\conv\{\mathbf p_1,\mathbf p_2,\mathbf p_3\}\subseteq A$, and hence in $\intr A$.

The boundary lines of $G^1,G^2,\widetilde G^3$ pass through the interior point $O$ and divide $A$ into six sectors of positive area, alternately covered by one and by two of these halfplanes. For $j=1,2,3$, let $\sigma_j$ be the area of the sector covered only by the $j$th halfplane and $\tau_j$ the area of the opposite sector, which is covered by the other two (\cref{fig:sectors}(c)). Since each of the three halfplanes covers a part of $A$ of area at most $\beta$, counting multiplicities gives $\sum_{j=1}^3(\sigma_j+2\tau_j)\le3\beta$. On the other hand, Tsintsifas's inequality~\cite[Theorem~1]{Tsintsifas}, applied to these three lines, gives $\sum_{j=1}^3\tau_j/\sigma_j\ge3/2$.

If $\sigma_j^2>\beta\tau_j$ for all $j=1,2,3$, then $\rho_j=\tau_j/\sigma_j$ satisfies $\sigma_j>\beta\rho_j$ and $\sum_{j=1}^3\rho_j\ge3/2$. The Cauchy--Schwarz inequality gives $\sum_{j=1}^3\rho_j^2\ge3/4$, so
\[
3\beta\ge\sum_{j=1}^3(\sigma_j+2\tau_j)>\beta\sum_{j=1}^3(\rho_j+2\rho_j^2)\ge3\beta,
\]
a contradiction. Hence $\sigma_j^2\le\beta\tau_j$ for some $j$, which is the required inequality for the two halfplanes other than the $j$th. It remains valid when $\widetilde G^3$ is replaced by $G^3$, since this can only decrease $\sigma$ and increase $\tau$.
\end{proof}

\section{Proof of the main theorem}\label{sec:proof}

\begin{proof}[Proof of \cref{thm:main}]
If $|A_i|>1/2$ for $i=0$ or $2$, Gr\"unbaum's inequality~\cite{Grunbaum} at the centroid of $A_i$ gives a point of depth at least $\frac49|A_i|>2/9$, so we may assume $|A_0|,|A_2|\le1/2$ and seek a point of depth at least $2/9$ in the middle slice. Suppose for a contradiction that $h_S((1,\x))<2/9$ for every $\x\in A_1$. Each such point lies in a closed halfspace of mass less than $2/9$ whose boundary can be moved slightly outward to give an open halfspace of mass less than $2/9$, by continuity from above. Compactness of $A_1$ then gives a finite cover of $\{1\}\times A_1$ by open halfspaces $H^1,\ldots,H^m$, each of mass less than $2/9$. Each $H^j$ meets the middle slice and is therefore not bounded by a plane parallel to the slices, since such a halfspace would contain the middle slice together with an outer slice and thus have mass at least $1/2$. Write $H_i^j=\{\x\in\R^2:(i,\x)\in H^j\}$ for $i=0,1,2$. Then $H_1^1,\ldots,H_1^m$ are open halfplanes covering $A_1$, with $|A_1\cap H_1^j|\le\mu(S\cap H^j)<2/9$.

\begin{figure}[t]
\centering
\begin{tikzpicture}[x=1cm,y=1cm,font=\small,line cap=round,line join=round,
  body/.style={draw=ink,line width=0.75pt},
  cut/.style={draw=ink,line width=0.55pt,postaction={decorate,draw=ink,line width=0.3pt,
    decoration={border,angle=-45,amplitude=0.85mm,segment length=1.3mm}}}]
  \def\Azero{(1.5,-0.8) -- (1.5,0.9) -- (-1.2,0.9) -- (-1.5,-0.8) -- cycle}
  \def\Aone{(1.6,-0.85) -- (1.6,0.95) -- (0.7,1.55) -- (-0.75,1.55) -- (-1.55,0.95) -- (-1.55,-0.85) -- cycle}
  \def\Atwo{(1.6,-0.8) -- (0.1,1.4) -- (-1.5,-0.8) -- cycle}
  \def\Kzero{(-0.7,0.9) -- (0.375,0.09375) -- (1.45,0.9) -- cycle}
  \def\Ktwo{({-131/170},{69/340}) -- (-0.375,-0.09375) -- ({233/380},{123/190}) -- (0.1,1.4) -- cycle}
  \def\Kmean{(0,0) -- ({98/95},{147/190}) -- (0.775,1.15) -- (-0.3,1.15) -- ({-25/34},{75/136}) -- cycle}
  \def\Uncovered{(0,0) -- ({121/85},{363/340}) -- (0.7,1.55) -- (-0.75,1.55) -- ({-169/120},{169/160}) -- cycle}

  \begin{scope}
    \node at (0,1.98) {$i=0$};
    \fill[black!3] \Azero;
    \fill[singleton!45] \Kzero;
    \draw[cut] (-1.0,1.125) -- (1.7,-0.9);
    \draw[cut] (-1.45,-1.275) -- (1.65,1.05);
    \draw[body] \Azero;
    \node at (0,-0.58) {$A_0$};
    \node at (0.375,0.61) {$K_0$};
  \end{scope}
  \begin{scope}[shift={(4.65,0)}]
    \node at (0,1.98) {$i=1$};
    \fill[black!3] \Aone;
    \fill[singleton!18] \Uncovered;
    \fill[singleton!45] \Kmean;
    \draw[cut] (-1.7,1.275) -- (1.7,-1.275);
    \draw[cut] (-1.7,-1.275) -- (1.7,1.275);
    \draw[body] \Aone;
    \draw[ink,densely dashed,line width=0.8pt] \Kmean;
    \node at (0,-0.61) {$A_1$};
    \node at (0.02,1.37) {$K_1$};
    \node[inner sep=0pt,font=\footnotesize] at (0.22,0.70) {$\dfrac{K_0+K_2}{2}$};
  \end{scope}
  \begin{scope}[shift={(9.3,0)}]
    \node at (0,1.98) {$i=2$};
    \fill[black!3] \Atwo;
    \fill[singleton!45] \Ktwo;
    \draw[cut] (-1.65,0.8625) -- (1.35,-1.3875);
    \draw[cut] (-1.7,-1.0875) -- (1.05,0.975);
    \draw[body] \Atwo;
    \node at (-0.375,-0.60) {$A_2$};
    \node at (-0.05,0.60) {$K_2$};
  \end{scope}
\end{tikzpicture}
\caption{The sets $K_i=A_i\setminus(H_i^1\cup H_i^2)$. Convexity of the complement of $H^1\cup H^2$ and the midpoint inclusion for the $A_i$ give $\frac12(K_0+K_2)\subseteq K_1$. The set $\frac12(K_0+K_2)$ is dashed, and $|K_1|=\sigma$.}
\label{fig:trimming}
\end{figure}
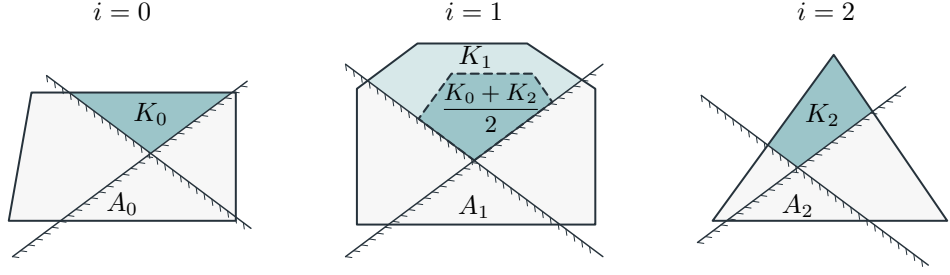

Apply \cref{lem:covering} to $A_1$ and these halfplanes with $\beta=2/9$, relabel the halfspaces $H^r,H^s$ given by the lemma as $H^1,H^2$, and set $K_i=A_i\setminus(H_i^1\cup H_i^2)$ for $i=0,1,2$. With $u=|K_0|$, $\sigma=|K_1|$, $v=|K_2|$, and $\tau=|A_1\cap H_1^1\cap H_1^2|$ (\cref{fig:trimming}), the lemma gives $\tau\ge\frac92\sigma^2$, while inclusion--exclusion yields
\[
\frac49>\mu(S\cap H^1)+\mu(S\cap H^2)=\mu(S\cap(H^1\cup H^2))+\mu(S\cap H^1\cap H^2)\ge1-u-\sigma-v+\tau.
\]
Consequently,
\begin{equation}\label{eq:uncovered}
u+v>\frac59+\tau-\sigma\ge\frac12+\frac92\left(\sigma-\frac19\right)^2.
\end{equation}

Because the complement of $H^1\cup H^2$ is closed and convex, the sets $K_0,K_1,K_2$ inherit the midpoint inclusion $\frac12(K_0+K_2)\subseteq K_1$. By \eqref{eq:uncovered} and the bounds $u\le|A_0|\le1/2$ and $v\le|A_2|\le1/2$, both $u$ and $v$ are positive, so the Brunn--Minkowski inequality applied to $K_0,K_2$ gives $\sqrt u+\sqrt v\le2\sqrt\sigma$. Moreover, $(1-2u)(1-2v)\ge0$ and \eqref{eq:uncovered} imply $4uv\ge2(u+v)-1>9(\sigma-1/9)^2$, hence $2\sqrt{uv}>3\sigma-1/3$. Squaring the Brunn--Minkowski bound and using \eqref{eq:uncovered} again, we obtain
\[
4\sigma\ge u+v+2\sqrt{uv}>4\sigma+\frac92\left(\sigma-\frac29\right)^2\ge4\sigma,
\]
a contradiction. For the mixed-integer conclusion, the three slices of positive area are consecutive, so translating them to $z=0,1,2$ and scaling the planar coordinates gives the hypotheses of the theorem by convexity.
\end{proof}

\pagebreak[3]
\section{Conclusions}\label{sec:conclusions}

We proved that if planar convex bodies $A_0,A_1,A_2$ satisfy $\frac12(A_0+A_2)\subseteq A_1$, then the union of the slices $\{i\}\times A_i\subseteq\R^3$ contains a point of halfspace depth at least $2/9$ of the total area, and that this constant is sharp. This settles the special case of Oertel's conjecture in $\Z\times\R^2$ with exactly three slices of positive area. The conjecture remains open in general, and our argument does not extend directly to more slices or to higher dimensions. With three slices, the midpoint inclusion allows the Brunn--Minkowski inequality to compare the areas that two halfspaces leave uncovered in the three slices, whereas four or more slices would require estimates for the uncovered areas of the additional slices. In $\Z\times\R^d$ with $d\ge3$, one would need an analogue of \cref{lem:covering} for the at most $d+1$ halfspaces supplied by Helly's theorem, and we are not aware of a higher-dimensional counterpart of Tsintsifas's inequality.

\paragraph{Acknowledgments.}
Both authors gratefully acknowledge support from the Air Force Office of Scientific Research (AFOSR) grant FA9550-25-1-0038 and the National Science Foundation (NSF) grant CMMI-2550265.

\paragraph{AI disclosure.}
The proof of \cref{thm:main} was obtained mainly by GPT-5.6 Sol~\cite{GPT56} and GPT-6 Astra~\cite{GPT6Astra} in Codex, using a fact graph system similar to Danus~\cite{Danus}. The resulting graph, available at \url{https://github.com/Hongyu-Cheng/oertel-three-slices-2d}, contains 651 facts, and the proof of \cref{thm:main} in it uses 31 of them and has depth 12. The authors rewrote this proof as the argument of \cref{sec:planar,sec:proof}, verified it, and take full responsibility for the paper.

\bibliographystyle{plainnat}
\bibliography{references}
\end{document}